\documentclass[a4paper,10pt,twoside,leqno]{article}

\usepackage{amsmath,amsthm,amsfonts,latexsym,amscd,amssymb}
\usepackage{hyperref}
\hypersetup{
	colorlinks=true,
	urlcolor=black,
	citecolor=blue}

\usepackage[a4paper, total={6.1in,8.6in}]{geometry}
\numberwithin{equation}{section}
\input xypic
\def\qed{{\hbadness=10000\hfill\ \vbox{\hrule height.09ex
			\hbox{\vrule width.09ex height1.55ex depth.2ex \kern1.8ex
				\vrule width.09ex height1.55ex depth.2ex}\hrule height.09ex}\break
		\bigskip}}
\newtheorem{theorem}{Theorem}[section]
\newtheorem{lemma}[theorem]{Lemma}
\newtheorem{corollary}{Corollary}[section]
\newtheorem{proposition}{Proposition}[section]

\theoremstyle{definition}

\newtheorem{example}[theorem]{Example}

\theoremstyle{remark}
\newtheorem{remark}[theorem]{Remark}
  
\newcommand{\n}{\noindent}

\begin{document}
	
	\linespread{1}\title {\textbf{Riemannian Submersion whose Total Manifold admits an Almost Hyperbolic Ricci Soliton}}
	
	\author{Sonali Jaiswal, Akhilesh Yadav, S.S. Shukla}
	\date{}
	
	\maketitle 
	
	\noindent\textbf{Abstract:} 
	In the present paper, we investigate geometric properties of a Riemannian submersion for which the total manifold is endowed with hyperbolic and almost hyperbolic Ricci solitons.  We establish conditions under which the fibers of the submersion admit hyperbolic and almost hyperbolic Ricci solitons. Further, we derive necessary conditions for the base manifold to be an Einstein and a hyperbolic Ricci soliton. Also, we determine the scalar curvature associated with the fibers and the base manifold. Moreover, when the total space admits an almost hyperbolic Ricci soliton with a torse-forming potential vector field, we obtain conditions ensuring that the extrinsic vertical and horizontal distributions are almost generalized quasi-Einstein. We also obtain necessary conditions for the fibers and base manifold to be an Einstein and almost generalized quasi-Einstein manifold. Finally, examples are included to illustrate the existence of hyperbolic and almost hyperbolic Ricci solitons in Riemannian submersion.
	\maketitle
	
	\n\textbf {Mathematics Subject Classification (2020):} 
	53C25, 53C40, 53E20, 53E40. 
	
	\n\textbf{Key words:}
	Riemannian submersion, hyperbolic Ricci soliton,	almost hyperbolic Ricci soliton, Einstein manifold, generalized quasi-Einstein manifold.

	\section{Introduction}
	
	Geometric flows play an important role in modern differential geometry because of their connections with curvature, topology, analysis, and mathematical physics. Among various geometric flows, the Ricci flow introduced by Hamilton \cite{rs} has become a fundamental tool for studying the evolution of Riemannian metrics and curvature. Ricci solitons arise as self-similar solutions of the Ricci flow and provide important generalizations of Einstein manifolds. The geometry of Einstein manifolds and their various generalizations has been extensively studied \cite{al}. In particular, quasi-Einstein and generalized quasi-Einstein manifolds were investigated in \cite{mc,mc1}, while super quasi-Einstein manifolds were studied in \cite{co}. Several geometric properties of Ricci solitons associated with special vector fields have also been investigated. In particular, concurrent, concircular, and torqued vector fields in relation to Ricci solitons were studied in \cite{by,by1,by3}, while further properties of concircular vector fields were considered in \cite{by2}. The notion of torse-forming vector fields, which is also relevant to the present study, originates from the work of Yano \cite{ky}.
	
	Another important development in the theory of geometric flows is the hyperbolic geometric flow, which has a wave-like character in contrast to the parabolic nature of the Ricci flow. Kong and Liu investigated the wave character of metrics and hyperbolic geometric flow \cite{dk}, while Dai, Kong, and Liu established short-time existence and nonlinear stability results for hyperbolic geometric flow \cite{wr}. The hyperbolic geometric flow was further studied on Riemann surfaces by Kong, Liu, and Xu \cite{dk1}. A dissipative hyperbolic geometric flow was subsequently investigated by Dai, Kong, and Liu \cite{wr1}. Self-similar solutions associated with hyperbolic geometric flow, known as hyperbolic Ricci solitons, have subsequently attracted considerable attention. Three-dimensional homogeneous hyperbolic Ricci solitons were studied in \cite{hf}, while hyperbolic Ricci solitons on warped product manifolds were investigated in \cite{sa}. Further results on hyperbolic Ricci solitons were obtained in \cite{am,gf}. Hyperbolic Ricci solitons and their gradient versions have also been studied in relativistic perfect fluid spacetimes \cite{md}. More recently, almost hyperbolic Ricci solitons were investigated on $(LCS)_n$-manifolds \cite{mj}.
	
	Riemannian submersions provide a natural framework for relating the geometry of a total manifold with that of its base and fibers. The fundamental equations of Riemannian submersions were established by O'Neill \cite{bo}, and the theory has been developed extensively in \cite{mf,bs}. The interaction between Riemannian submersions and Ricci solitons has been investigated by Meri\c{c} and Kili\c{c}, who obtained conditions under which the fibers and the base manifold inherit Ricci soliton or almost Ricci soliton structures \cite{em}. They also studied Riemannian submersions whose total spaces are endowed with torse-forming vector fields and obtained further geometric characterizations \cite{em1}. Various generalized soliton structures have also been studied in the context of Riemannian submersions. For example, $\eta$-Ricci-Yamabe solitons and almost $\eta$-Ricci-Bourguignon solitons were considered in \cite{ms,sk}, while conformal $\eta$-Ricci solitons on Riemannian submersions were investigated in \cite{md1}. Almost Yamabe solitons on total spaces of almost Hermitian submersions were studied in \cite{tf}. These investigations demonstrate the relevance of submersion theory in the study of generalized soliton structures.
	
	Conformal submersions provide a further generalization of Riemannian submersions and have also been studied in connection with Ricci soliton geometry. Meena and Yadav investigated conformal submersions whose total manifolds admit Ricci solitons and obtained conditions under which the fibers and the base manifold admit Ricci soliton, almost Ricci soliton, and Einstein structures \cite{km}. Further generalized Ricci-type soliton structures have been considered in different geometric settings, including $\eta$-Ricci solitons and contact geometry \cite{sd}, as well as almost conformal Ricci solitons on weakly Ricci symmetric spacetimes \cite{ts}.
	
	Motivated by these developments, we investigate several geometric properties associated with hyperbolic and almost hyperbolic Ricci soliton structures in the setting of Riemannian submersions. First, we obtain conditions under which the fibers admit hyperbolic or almost hyperbolic Ricci solitons. We also establish conditions under which the base manifold becomes Einstein or admits a hyperbolic Ricci soliton. Furthermore, we obtain conditions for the potential vector field to be conformal and determine the scalar curvature of the fibers and base manifold.
	
	Next, we consider a torse-forming potential vector field and derive further geometric consequences. Under suitable conditions, when the vertical and horizontal components of the potential vector field are concircular, we show that the corresponding fibers and base manifold become Einstein manifolds, respectively. We also establish conditions under which the extrinsic vertical and horizontal distributions admit an almost generalized quasi-Einstein structure. Finally, examples are provided to illustrate the obtained results.
	
	The paper is organized as follows. Section~2 contains the necessary preliminaries on Riemannian submersions and hyperbolic and almost hyperbolic Ricci solitons. Section~3 is devoted to the geometric properties associated with hyperbolic Ricci solitons, including results concerning the fibers, base manifold, potential vector field, and scalar curvature. In Section~4, we study almost hyperbolic Ricci solitons with torse-forming potential vector fields and obtain the corresponding Einstein and almost generalized quasi-Einstein structures. Finally, examples illustrating the obtained results are presented.

	\section{Preliminaries}
	
	This section is devoted to a brief review of the basic notions and fundamental results concerning Riemannian submersions, which are required in the subsequent sections.

	A Riemannian submersion \cite{bo}, is a smooth surjective mapping $\pi$ from an
	$m$-dimensional Riemannian manifold $(M^{m},g)$ onto an
	$n$-dimensional Riemannian manifold $(B^{n},g')$ satisfying the	following axioms:
	\begin{enumerate}
		\item[(i)]  The map $\pi$ is of maximal rank,
		\item[(ii)]  The differential map $\pi_{\ast}$ preserves the lengths of
		horizontal vectors.
 	\end{enumerate}
	
	Let $\pi:(M,g)\rightarrow(B,g')$ be a Riemannian submersion.
	If $X$ and $Y$ are basic vector fields on $M$, then the following
	properties holds \cite{bo},
	\begin{enumerate}
		\item [(i)] $g(X,Y)=g'(X',Y')\circ\pi$,
		\item [(ii)]$\mathcal{H}[X,Y]$ is a basic vector field and is $\pi$-related
		to $[X',Y']$,
		\item [(iii)] $\mathcal{H}(\nabla_{X}Y)$ is a basic vector field corresponding
		to $\nabla'_{X'}Y'$, where $\nabla'$ denotes the  Levi-Civita connection
		on $(B,g')$, 
		\item [(iv)]for any vertical vector field $V$, the Lie bracket $[X,V]$
		remains vertical.
	\end{enumerate}
	
	The fundamental $(1,2)$-tensor fields $T$ and $A$ associated with a
	Riemannian submersion on $(M,g)$ are defined by \cite{bo}
	\begin{equation*}
		T_{E}F=\mathcal{H}\nabla_{\mathcal{V}E}\mathcal{V}F+
		\mathcal{V}\nabla_{\mathcal{V}E}\mathcal{H}F,
	\end{equation*}
	\begin{equation*}
		A_{E}F=\mathcal{V}\nabla_{\mathcal{H}E}\mathcal{H}F+
		\mathcal{H}\nabla_{\mathcal{H}E}\mathcal{V}F,
	\end{equation*}
	for any vector fields $E$ and $F$ on $(M,g)$ and $\nabla$ denotes the
	Levi-Civita connection of $(M,g)$. These tensor fields are known as the
	integrability tensors of the Riemannian submersion. For
	$E\in\Gamma(TM)$, the operators $T_{E}$ and $A_{E}$ act as skew-symmetric
	endomorphisms on $\Gamma(TM)$ and interchange the vertical and
	horizontal distributions.
	
	We denote by $\mathcal{V}$ and $\mathcal{H}$ the projection morphisms
	onto the distributions $ker\pi_{\ast}$ and $(ker\pi_{\ast})^{\perp}$,
	respectively. It is straightforward to observe that $T$ is a vertical
	tensor, that is $T_{E}=T_{\mathcal{V}E}$, while $A$ is horizontal, that
	is $A_{E}=A_{\mathcal{H}E}$. Using these properties together with the
	metric compatibility condition $\nabla g=0$, we have the following
	results,
	
	\begin{lemma}$\cite{mf}$
		For any $X,Y,Z\in\Gamma(TM)$, the following relations hold,
		\begin{align*}
	g(T_{X}Y,Z)=-g(Y,T_{X}Z),\,\,\,\,\,\,\,\,\,\
			g(A_{X}Y,Z)=-g(Y,A_{X}Z).
		\end{align*}
	\end{lemma}
	
	\begin{lemma}$\cite{mf}$
		For any vertical vector fields $U,W$ and horizontal vector fields
		$X,Y$, the tensor fields $T$ and $A$ satisfy
		\begin{align*}
			T_{U}W = T_{W}U,\,\,\,\,\,\,\,\,\,\,
			A_{X}Y =-A_{Y}X= \frac{1}{2}\mathcal{V}[X,Y].
		\end{align*}
	\end{lemma}
	
	If $X$ is a basic vector field and $U$ is vertical, then
	$\mathcal{H}(\nabla_{U}X)=\mathcal{H}(\nabla_{X}U)=A_{X}U$, and the Lie
	bracket $[U,X]$ is vertical. The horizontal distribution
	$\mathcal{H}$ is integrable if and only if $A=0$, and hence $A$ is called
	the integrability tensor of the Riemannian submersion. The tensor field $T$ coincides with the second fundamental form of the
	fibers. Consequently, the condition $T=0$ is equivalent to the fibers
	being totally geodesic. A Riemannian submersion is said to have totally
	geodesic fibers if $T$ vanishes identically.
	
	Let $\overline{\nabla}$ denote the Schouten connection \cite{mf}, associated with the
	mutually orthogonal horizontal and vertical distributions. For any
	$U,W\in\Gamma(\mathcal{V})$ and $X,Y\in\Gamma(\mathcal{H})$, we have
	\begin{equation}\label{2.1}
		\nabla_{U}W=T_{U}W+\overline{\nabla}_{U}W,
	\end{equation}
	\begin{equation}\label{2.2}
		\nabla_{U}X=\mathcal{H}\nabla_{U}X+T_{U}X,
	\end{equation}
	\begin{equation}\label{2.3}
		\nabla_{X}W=A_{X}W+\mathcal{V}\nabla_{X}W,
	\end{equation}
	\begin{equation}\label{2.4}
		\nabla_{X}Y=\mathcal{H}\nabla_{X}Y+A_{X}Y,
	\end{equation}
	where
	$\overline{\nabla}_{U}W=\mathcal{V}\nabla_{U}W=\widehat{\nabla}_{U}W$.
	\begin{lemma}$\cite{mf}$
		For a Riemannian submersion, the following statements hold:
		\begin{enumerate}
			\item [(i)] if the tensor $A$ is parallel, then $A$ vanishes,
			\item [(ii)] if the tensor $T$ is parallel, then $T$ vanishes.
		\end{enumerate}
	\end{lemma}
	Let $R^\ast$ denote the $(1,3)$-tensor field on $\Gamma(\mathcal{H})$
	and let $R'$ be the Riemannian curvature tensor of $(B,g')$. For any
	horizontal vector fields $X,Y,Z,J$, we have
	\begin{equation*}
		\pi_\ast (R^\ast (X,Y,Z))=R'(\pi_\ast X,\pi_\ast Y,\pi_\ast Z),
	\end{equation*}
	\begin{equation*}
		R^\ast (X,Y,Z,J)=R'(\pi_\ast X,\pi_\ast Y,\pi_\ast Z,\pi_\ast J)\circ\pi.
	\end{equation*}
	
	Further, let $\widehat{R}$ denote the Riemannian curvature tensor of a fiber. By applying the Gauss and Codazzi equations, we have
	\begin{align}\label{2.5}
		R(U,V,W,P)=&\widehat{R}(U,V,W,P)
		-g(T_V W,T_U P)+g(T_U W,T_V P),\\
		R(X,Y,Z,J)=&R^\ast (X,Y,Z,J)
		-g(A_Y Z,A_X J)-g(A_Z X, A_Y J)\\&
		+2g(A_X Y, A_Z J),\notag
	\end{align}
	for any vertical vector fields $U,V,W,P$ and horizontal vector fields $X,Y,Z,J$.	We call a local orthonormal frame	$\{X_i,U_j\}_{1\leq i\leq n,\,1\leq j\leq r}$ on $M$ a $\pi$-adapted frame,
	if each $X_i$ is horizontal and each $U_j$ is vertical. Let $N$ be the
	horizontal vector field on $M$ locally defined by
	\begin{align}\label{2.7}
		N=\sum_{j=1}^{r}T_{U_j}U_j=rH,
	\end{align}
	where $\{U_j\}_{1\leq j\leq r}$ is a local orthonormal frame of the
	vertical distribution $\mathcal{V}$ and $H$ denotes the mean curvature
	vector field of the fibers. A Riemannian submersion is said to be
	minimal, if $H=0$.
	
	A Riemannian submersion is said to have totally umbilical fibers, if
	\begin{equation}\label{2.8}
		T_{U}W=g(U,W)H,
	\end{equation}
	for all vertical vector fields $U,W$.
	
	Also, for any vector field $E\in\Gamma(TM)$ we have
	\begin{equation}\label{2.9}
		g(\nabla_E N,X)=\sum_{j=1}^{r}g((\nabla_E T)_{U_j}U_j,X).
	\end{equation}
	We denote by $\tilde{\delta}(X)$ the horizontal divergence of a vector
	field $X\in\Gamma(\mathcal{H})$, which is defined as
	\begin{align}\label{2.10}
		\widetilde{\delta}(X)=\sum_{i=1}^{n}g(\nabla_{X_i}X,X_i),
	\end{align}
	where $\{X_1,X_2,\ldots,X_n\}$ is a local orthonormal basis of the
	horizontal distribution. Consequently, we have
	\begin{equation}\label{2.11}
		\widetilde{\delta}(N)=\sum_{i=1}^{n}\sum_{j=1}^{r}
		g((\nabla_{X_i}T)_{U_j}U_j,X_i),
	\end{equation}
	
	for an orthonormal $\pi$-adapted frame 
	$\{X_i,U_j\}_{1\leq i\leq n,\,1\leq j\leq r}$ on the total manifold
	$(M,g)$. The Ricci curvature tensor $Ric$ and the scalar curvature $\tau$ on	$M$ is given as
	\begin{align}\label{2.12}
		\operatorname{Ric}(U,W)=&\widehat{\operatorname{Ric}}(U,W)-g(T_U W,N)
		+\sum_{i=1}^{n}\Big[g((\nabla_{X_i}T)_U W,X_i)\\
		&+g(A_{X_i}U,A_{X_i}W)\Big],\notag
	\end{align}
	\begin{align}\label{2.13}
		\operatorname{Ric}(X,Y)=&\operatorname{Ric}'(X',Y')\circ\pi
		+\frac{1}{2}\Big[g(\nabla_X N,Y)+g(\nabla_Y N,X)\Big]\\
		&-2\sum_{i=1}^{n}g(A_X X_i,A_Y X_i)
		-\sum_{j=1}^{r}g(T_{U_j}X,T_{U_j}Y),\notag
	\end{align}
	\begin{align}\label{2.14}
		\operatorname{Ric}(U,X)=&g(\nabla_U N,X)
		-\sum_{j=1}^{r}g((\nabla_{U_j}T)_{U_j}U,X)\\
		&+\sum_{i=1}^{n}\Big[g((\nabla_{X_i}A)_{X_i}X,U)
		-2g(A_X X_i,T_U X_i)\Big],\notag
	\end{align}
	\begin{align}\label{2.15}
		\tau=\tau'\circ\pi+\widehat{\tau}-(||N||^2+||A||^2+||T||^2)+2\sum_{i=1}^{n}g(\nabla_{X_i}N,X_i),
	\end{align}
	where $\operatorname{Ric}'$ and $\tau'$  are the Ricci curvature and the scalar curvature of the base manifold and $\widehat{\operatorname{Ric}}$ and $\widehat{\tau}$ are the Ricci curvature and the scalar curvature of fibers, respectively.
	
	A hyperbolic Ricci soliton is defined as a self-similar solution of the hyperbolic geometric flow. A systematic formulation of hyperbolic Ricci solitons was developed in \cite{hf}. Let	$(M,g)$ be a Riemannian manifold. The 5-tuple 
	$(M,g,{\zeta},\lambda,\mu)$ is called a hyperbolic Ricci soliton, if there exists a smooth vector field 
	${\zeta}$ and real constants $\lambda$ and $\mu$, for any vector fields $E,F\in\Gamma(TM)$ such that
	
	\begin{equation}\label{2.16}
		\frac{1}{2}(L_{\zeta} L_{\zeta} g)(E,F)+\lambda (L_{\zeta} g)(E,F)+\operatorname{Ric}(E,F)=\mu g(E,F),
	\end{equation}
	where $Ric$ denotes the Ricci curvature tensor of $(M,g)$ and $L_{\zeta} g$ is a Lie derivative of metric $g$ is given by
	\begin{align}\label{2.17}
		L_{\zeta} g(E,F)=g(\nabla_E {\zeta},F)+g(E,\nabla_F {\zeta}),
	\end{align}
	and the second Lie derivative $L_{\zeta} L_{\zeta} g$ is given by
	\begin{align}\label{2.18}
		(L_{\zeta} L_{\zeta} g)(E,F)=& g(\nabla_{\zeta} \nabla_E {\zeta},F)+g(\nabla_{\zeta} \nabla_F {\zeta},E)-g(\nabla_{[{\zeta},E]} {\zeta},F) \\\notag
		& -g(\nabla_{[{\zeta},F]} {\zeta},E)+2g(\nabla_E {\zeta},\nabla_F {\zeta}). 
	\end{align}
	If ${\zeta}$ is a Killing vector field (i.e., $L_{\zeta} g=0$), the manifold is reduced to an Einstein manifold. If ${\zeta}$ is a 2-Killing vector field with $L_{\zeta} L_{\zeta} g=0$ and $\lambda=\frac{1}{2}$, the hyperbolic Ricci soliton transforms into a classical Ricci soliton. A hyperbolic Ricci soliton is classified as \textit{shrinking}, \textit{steady}, or \textit{expanding} according as $\lambda<0$, $\lambda=0$, or $\lambda>0$, respectively. Likewise, its rate of change is referred to as \textit{shrinking}, \textit{steady}, or \textit{expanding} when $\mu<0$, $\mu=0$, or $\mu>0$, respectively. A hyperbolic Ricci soliton becomes an almost hyperbolic Ricci soliton when $\lambda$ and $\mu$  is taken as a smooth function on the manifold rather than a constant.
	Finally, for any vector field $E$ on a Riemannian manifold 
	$(M,g)$ is said to be conformal, if it satisfies
	\begin{equation*}
		L_E g=2\sigma g,
	\end{equation*}
	where $\sigma$ is a smooth function on $M$.
	The vector field $E$ is homothetic, when $\sigma$ is constant and is a Killing, when $\sigma=0$.
	
	\section{ Riemannian Submersion with Hyperbolic Ricci Soliton}
	In this section, we investigate a Riemannian submersion whose total space is endowed with a hyperbolic Ricci soliton and study the geometric characteristics arising from this submersion.
	
	\begin{lemma} $\cite{em}$
		Let $\pi:(M,g)\rightarrow(B,g')$ be Riemannian submersion between Riemannian manifolds. Then the following conditions are equivalent to each other:
		\begin{enumerate}
			\item [(i)] the vertical distribution $\mathcal{V}$ is parallel,
			\item [(ii)] the horizontal distribution $\mathcal{H}$ is parallel,
			\item [(iii)] the fundamental tensors $T$ and $A$ are vanishes identically, i.e. $T\equiv0$ and $A\equiv0$.
		\end{enumerate}
	\end{lemma}
	\begin{theorem}
		Let $\pi:(M,g,\xi,\lambda,\mu)\rightarrow (B,g')$ be a Riemannian submersion between Riemannian manifolds. The total manifold admits a hyperbolic Ricci soliton with the vertical potential field $\xi$ and vanishing fundamental tensors, then any fiber of $\pi$ admits a hyperbolic Ricci soliton.
		Moreover, if $\xi$ is a unit vector field, then the rate of change of soliton induced on each fiber is 
		 \textit{shrinking} if $\widehat{\tau}<0$,
		 \textit{steady} if $\widehat{\tau}=0$, and  
	 \textit{expanding} if $\widehat{\tau}>0$.
	\end{theorem}
	\begin{proof}
		Since $(M,g)$ admits a hyperbolic Ricci soliton  with vertical potential vector field $\xi$, for any $U,W\in\Gamma(\mathcal{V})$  we have
		\begin{equation}\label{3.1}
			\frac{1}{2}(L_\xi L_\xi g)(U,W)+\lambda (L_\xi g)(U,W)+\operatorname{Ric}(U,W)=\mu g(U,W).
		\end{equation} 
		Using (\ref{2.17}) and (\ref{2.18}) in the above equation, we get
		\begin{align}\label{3.2}
			\frac{1}{2} \Big[g(\nabla_\xi \nabla_U \xi,W)+g(\nabla_\xi \nabla_W \xi,U)-g(\nabla_{[\xi,U]} \xi,W) 
			-g(\nabla_{[\xi,W]} \xi,U)\\\notag+2g(\nabla_U \xi,\nabla_W \xi)\Big] +\lambda\Big[g(\nabla_U \xi,W)+g(U,\nabla_W \xi)\Big]\\
			+\operatorname{Ric}(U,W)=\mu g(U,W),\notag
		\end{align}
		where $\{X_i\}$ denotes the orthonormal basis of the horizontal distribution $\mathcal{H}$. 	Since the  fundamental tensors vanish identically on the manifold and using (\ref{2.1}), (\ref{2.2}) and (\ref{2.12}) in (\ref{3.2}), we get 
		\begin{align}\label{frb}
			\frac{1}{2} \Big[g(\widehat{\nabla}_\xi \widehat{\nabla}_U \xi,W)+g(\widehat{\nabla}_\xi \widehat{\nabla}_W \xi,U)-g(\widehat{\nabla}_{[\xi,U]} \xi,W) 
			-g(\widehat{\nabla}_{[\xi,W]} \xi,U)\\+2g(\widehat{\nabla}_U \xi,\widehat{\nabla}_W \xi)\Big]\notag +\lambda\Big[g(\widehat{\nabla}_U \xi,W)+g(U,\widehat{\nabla}_W \xi)\Big]\\ +\widehat{\operatorname{Ric}}(U,W)=\mu g(U,W).\notag 
		\end{align}
		Hence, any fibre of the Riemannian submersion admits a hyperbolic Ricci soliton. 	Since $\xi$ is unit vector field and by tracing $\eqref{frb}$, we obtain 	
		$\widehat{\tau}=r\mu.$ Consequently, the rate of change of soliton on induced fibers is shrinking if 	$\widehat{\tau}<0$, steady if 	$\widehat{\tau}=0$  and expanding if 	$\widehat{\tau}>0$, respectively.
		
	\end{proof}
	\begin{theorem}
		Let $\pi:(M,g,\eta,\lambda,\mu)\rightarrow (B,g')$ be a Riemannian submersion between Riemannian manifolds. The total manifold admits a hyperbolic Ricci soliton with the horizontal potential field $\eta$ and vanishing fundamental tensors, then any fiber of $\pi$ is an Einstein manifold if and only if $\nabla\eta$ is vertical.
	\end{theorem}
	\begin{proof}
		Since the total space $(M,g)$ admits a hyperbolic Ricci soliton  with the horizontal potential field $\eta$, for any $U,W\in\Gamma(\mathcal{V})$ we have
		\begin{equation*}
			\frac{1}{2}(L_\eta L_\eta g)(U,W)+\lambda (L_\eta g)(U,W)+\operatorname{Ric}(U,W)=\mu g(U,W).
		\end{equation*} 
		Using (\ref{2.17}) and (\ref{2.18}) in the above equation, we get
		\begin{align*}
			\frac{1}{2} \Big[g(\nabla_\eta \nabla_U \eta,W)+g(\nabla_\eta \nabla_W \eta,U)-g(\nabla_{[\eta,U]} \eta,W) 
			-g(\nabla_{[\eta,W]} \eta,U)\\\notag+2g(\nabla_U \eta,\nabla_W \eta)\Big] +\lambda\Big[g(\nabla_U \eta,W)+g(U,\nabla_W \eta)\Big]\\
			+\operatorname{Ric}(U,W)=\mu g(U,W).
		\end{align*}
		Again using (\ref{2.12}) in the preceding equation, we obtain
		\begin{align}\label{3.3}
			\frac{1}{2}\, \Big[g(A_{\eta} (\mathcal{H}\nabla_{U} \eta),W) +g(\mathcal{V}\nabla_{\eta} (T_{U} \eta),W) + g(A_{\eta} (\mathcal{H}\nabla_{W} \eta),U)\\ +g(\mathcal{V}\nabla_{\eta} (T_{W} \eta),U)
			-g(T_{[\eta,{U}]}\eta,{W}) \notag
			-g(T_{[\eta,{W}]}\eta,{U})
			+2g(T_{U} \eta,T_{W}\notag \eta)\\ \notag+2g(\mathcal{H}\nabla_{U} \eta,\mathcal{H}\nabla_{W} \eta)\Big] +\lambda\Big[g(T_U\notag \eta,W)+g(U,T_W \eta)\Big]\notag
			+\widehat{\operatorname{Ric}}(U,W)\\-\sum_{j=1}^{r}g(T_U W,T_{U_j}{U_j})+\notag
			\sum_{i=1}^{n}\Big[g(\nabla_{X_i} H,X_i)g(U,W)
			+g(A_{X_i}U,A_{X_i}W)\Big]\\=\mu g(U,W).\notag
		\end{align}
		Since fundamental tensors vanish identically on the manifold,  from the previous equation we get
		\begin{align*}
			g(\mathcal{H}{\nabla}_U \eta,\mathcal{H}{\nabla}_W \eta)+\widehat{\operatorname{Ric}}(U,W)=\mu g(U,W). 
		\end{align*}
		Hence, any fiber of $\pi$ is an Einstein manifold if and only if $\nabla\eta$ is vertical.  
	\end{proof}
	\begin{theorem}\label{T3.4}
		Let $\pi:(M,g,\xi,\lambda,\mu)\rightarrow (B,g')$ be a Riemannian submersion between Riemannian manifolds. The total manifold admits a hyperbolic Ricci soliton with the vertical potential field $\xi$ and integrable  horizontal distribution. If any fiber of $\pi$ is totally umbilical and has a constant mean curvature, then it admits a hyperbolic Ricci soliton.
	\end{theorem}
	\begin{proof}
		Since the total space $(M,g)$ of Riemannian submersion $\pi$ admits a hyperbolic Ricci soliton  with vertical potential vector field $\xi$, by using (\ref{2.1}) and (\ref{2.2}) in (\ref{3.2}), we obtain
		\begin{align}\label{3.4}
			\frac{1}{2} \Big[g(\widehat{\nabla}_\xi \widehat{\nabla}_U \xi,W)+g(\widehat{\nabla}_\xi \widehat{\nabla}_W \xi,U)-g(\widehat{\nabla}_{[\xi,U]} \xi,W) 
			-g(\widehat{\nabla}_{[\xi,W]} \xi,U)\\\notag+2g(\widehat{\nabla}_U \xi,\widehat{\nabla}_W \xi)\Big]+\lambda\Big[g(\widehat{\nabla}_U \xi,W)+g(U,\widehat{\nabla}_W \xi)\Big]\\
			+\operatorname{Ric}(U,W)=\mu g(U,W), \notag
		\end{align}
		for any $U,W \in\Gamma(\mathcal{V})$. The manifold $(M,g)$ has totally umbilical fibers, therefore using (\ref{2.12}) in (\ref{3.4}), we get
		\begin{align*}
			\frac{1}{2} \Big[g(\widehat{\nabla}_\xi \widehat{\nabla}_U \xi,W)+g(\widehat{\nabla}_\xi \widehat{\nabla}_W \xi,U)-g(\widehat{\nabla}_{[\xi,U]} \xi,W) 
			-g(\widehat{\nabla}_{[\xi,W]} \xi,U)\\+2g(\widehat{\nabla}_U \xi,\widehat{\nabla}_W \xi)\Big]+\lambda\Big[g(\widehat{\nabla}_U \xi,W)+g(U,\widehat{\nabla}_W \xi)\Big]+\widehat{\operatorname{Ric}}(U,W)\\-\sum_{j=1}^{r}g(T_U W,T_{U_j}{U_j})+
			\sum_{i=1}^{n}\Big[g(\nabla_{X_i} H,X_i)g(U,W)
			\\+g(A_{X_i}U,A_{X_i}W)\Big]=\mu g(U,W). 
		\end{align*}
		Since the horizontal distribution $\mathcal{H}$ is integrable, from the preceding equation 
		\begin{align*}
			\frac{1}{2} \Big[g(\widehat{\nabla}_\xi \widehat{\nabla}_U \xi,W)+g(\widehat{\nabla}_\xi \widehat{\nabla}_W \xi,U)-
			g(\widehat{\nabla}_{[\xi,U]} \xi,W)-g(\widehat{\nabla}_{[\xi,W]} \xi,U)\\+ 2g(\widehat{\nabla}_U \xi,\widehat{\nabla}_W \xi)\Big] +\lambda\Big[g(\widehat{\nabla}_U \xi,W)+g(U,\widehat{\nabla}_W \xi)\Big]+\widehat{\operatorname{Ric}}(U,W)\\-g(U,W)g(H,N)+	\sum_{i=1}^{n}g(\nabla_{X_i} H,X_i)g(U,W)=\mu g(U,W), 
		\end{align*}
		where $H$ is a mean curvature vector of any fiber of $\pi$, which is constant. So from the above equation we obtain 
		\begin{equation*}
		\frac{1}{2}(L_\xi L_\xi g)(U,W)+\lambda (L_\xi g)(U,W) +\widehat{\operatorname{Ric}}(U,W)=(\mu+r||H||^2)g(U,W).
		\end{equation*}
		Thus any fiber of Riemannian submersion is  hyperbolic Ricci soliton. 
	\end{proof}

	\begin{theorem}
		Let $\pi:(M,g,\xi,\lambda,\mu)\rightarrow (B,g')$ be a Riemannian submersion between Riemannian manifolds. The total manifold admits a hyperbolic Ricci soliton with the vertical potential field $\xi$ and vanishing fundamental tensors, then $(B,g')$ is an Einstein manifold if and only if $\nabla\xi$ is horizontal.
	\end{theorem}
	\begin{proof}
		The total space $M$ of Riemannian submersion $\pi$ admits a hyperbolic Ricci soliton  with the vertical potential field $\xi$, therefore for any $X,Y\in\Gamma(\mathcal{H})$ we have 
		\begin{equation}\label{3.5}
			\frac{1}{2}(L_\xi L_\xi g)(X,Y)+\lambda (L_\xi g)(X,Y)+\operatorname{Ric}(X,Y)=\mu g(X,Y).
		\end{equation} 
		Using (\ref{2.13}), (\ref{2.17}) and (\ref{2.18}) in (\ref{3.5}), we get
		\begin{align}\label{3.6}
			\frac{1}{2} \Big[g(\nabla_\xi \nabla_X \xi,Y)+g(\nabla_\xi \nabla_Y \xi,X)-g(\nabla_{[\xi,X]} \xi,Y) 
			-g(\nabla_{[\xi,Y]} \xi,X)+\\2g(\nabla_X \xi,\nabla_Y \xi)\Big]\notag+\lambda\Big[g(\nabla_X \xi,Y)+g(X,\nabla_Y \xi)\Big]
			+\operatorname{Ric}'(X',Y')\circ\pi\\+\notag\frac{1}{2}\Big[g(\nabla_X N,Y)+g(\nabla_Y N,X)\Big]
			-2\sum_{i=1}^{n}g(A_X X_i,A_Y X_i)\\-\sum_{j=1}^{r}g(T_{X_j} X,T_{X_j} Y)=\mu g(X,Y),\notag
		\end{align}
		where $X'$ and $Y'$ are $\pi$-related to $X$ and $Y$, respectively. Now, using (\ref{2.1}), (\ref{2.2}) and (\ref{2.3}) in (\ref{3.6}) and get
		\begin{align*}
			\frac{1}{2} \Big[ g(\mathcal{H}\nabla_\xi(A_X \xi),Y)+g(T_\xi(\mathcal{V}\nabla_X \xi),Y)+g(\mathcal{H}\nabla_\xi(A_Y \xi),X)+g(T_\xi(\mathcal{V}\nabla_Y \xi),X) \\ -g(T_{[\xi,X]}\xi,Y) - g(T_{[\xi,Y]}\xi,X)+2g(A_X \xi,A_Y \xi)+2g(\mathcal{V}\nabla_X \xi,\mathcal{V}\nabla_Y \xi)\Big] \\+\lambda\Big[g(A_X \xi,Y)+g(X,A_Y \xi)\Big]
			+\operatorname{Ric}'(X',Y')\circ\pi+\frac{1}{2}\Big[g(\nabla_X N,Y)\\+g(\nabla_Y N,X)\Big] 
			-2\sum_{i=1}^{n}g(A_X X_i,A_Y X_i)-\sum_{j=1}^{r}g(T_{X_j} X,T_{X_j} Y)=\mu g(X,Y).
		\end{align*}
		Since  fundamental tensors vanish identically on the manifold,  the above equation yields
		\begin{align*}
			g(\mathcal{V}\nabla_X \xi,\mathcal{V}\nabla_Y \xi)+\operatorname{Ric}'(X',Y')\circ\pi=\mu g(X,Y),
		\end{align*}  
		where the horizontal vector fields $X,Y$ are $\pi$-related to $X',Y'$ respectively,
		which is equivalent to
		\begin{align*}
			g(\mathcal{V}\nabla_X \xi,\mathcal{V}\nabla_Y \xi)+	\operatorname{Ric}'(X',Y')=\mu g'(X',Y').
		\end{align*}
		Hence $(B,g')$ is an Einstein manifold if and only if $\nabla \xi$ is horizontal.
	\end{proof}
	\begin{theorem}
		Let $\pi:(M,g,\eta,\lambda,\mu)\rightarrow (B,g')$ be a Riemannian submersion between Riemannian manifolds. The total manifold admits a hyperbolic Ricci soliton with the horizontal potential field $\eta$ and vanishing fundamental tensors, then	$(B,g')$ admits hyperbolic Ricci soliton with potential field $\eta'$.  Moreover, if $\eta$ is a unit vector field on the base manifold, then the rate of change of corresponding soliton on the base manifold is 
		 \textit{shrinking} if ${\tau'}<0$,
	\textit{steady} if ${\tau'}=0$, and
	 \textit{expanding} if ${\tau'}>0$.
	\end{theorem}
	\begin{proof}
		Since the total space $(M,g)$ of Riemannian submersion $\pi$ admits a hyperbolic Ricci soliton  with the horizontal potential field $\eta$, 	for any vector fields $X,Y\in\Gamma(\mathcal{H})$ we have 
		\begin{equation}\label{3.7}
			\frac{1}{2}(L_\eta L_\eta g)(X,Y)+\lambda (L_\eta g)(X,Y)+\operatorname{Ric}(X,Y)=\mu g(X,Y).
		\end{equation} 
		Using (\ref{2.17}) and (\ref{2.18}) in (\ref{3.7}), we get
		\begin{align*}
			\frac{1}{2} \Big[g(\nabla_\eta \nabla_X \eta,Y)+g(\nabla_\eta \nabla_Y \eta,X)-g(\nabla_{[\eta,X]} \eta,Y) 
			-g(\nabla_{[\eta,Y]} \eta,X)\\+2g(\nabla_X \eta,\nabla_Y \eta)\Big]+\lambda\Big[g(\nabla_X \eta,Y)+g(X,\nabla_Y \eta)\Big]\notag
			\\	+\operatorname{Ric}(X,Y)=\mu g(X,Y).\notag
		\end{align*}
		Now, using (\ref{2.2}), (\ref{2.3}) and (\ref{2.4}) in the above equation, we get
		\begin{align}\label{3.8}
			\frac{1}{2} \Big[g(\nabla^{\ast}_\eta \nabla^{\ast}_X \eta,Y)+g(A_\eta A_X \eta,Y) +g(\nabla^{\ast}_\eta \nabla^{\ast}_Y \eta,X)+g(A_\eta A_Y \eta,X) 
			-\\ g(\mathcal{H}\nabla_{[\eta,X]}\eta,Y)\notag- g(\mathcal{H}\nabla_{[\eta,Y]}\eta,X)
			+2 g(\nabla^{\ast}_X \eta,\nabla^{\ast}_Y \eta)+2g(A_X \eta,A_Y \eta)\Big] \\+\lambda\Big[g(\nabla_X \eta,Y)+g(X,\nabla_Y \eta)\Big]
			+\operatorname{Ric}(X,Y)=\mu g(X,Y).\notag
		\end{align}
		Again, using (\ref{2.13}) in (\ref{3.8}) we obtain
		\begin{align*}
			\frac{1}{2} \Big[g(\nabla^{\ast}_\eta \nabla^{\ast}_X \eta,Y)+g(A_\eta A_X \eta,Y) +g(\nabla^{\ast}_\eta \nabla^{\ast}_Y \eta,X)+g(A_\eta A_Y \eta,X)\\ 
			- g(\mathcal{H}\nabla_{[\eta,X]}\eta,Y) - g(\mathcal{H}\nabla_{[\eta,Y]}\eta,X)
			+2 g(\nabla^{\ast}_X \eta,\nabla^{\ast}_Y \eta)+2g(A_X \eta,A_Y \eta)\Big]\\ +\lambda\Big[g(\nabla_X \eta,Y)+g(X,\nabla_Y \eta)\Big]
			+\operatorname{Ric}'(X',Y')\circ\pi+\frac{1}{2}\Big[ g(\nabla_X N,Y)\\+g(\nabla_Y N,X)\Big]
			-2\sum_{i=1}^{n}g(A_X X_i,A_Y X_i)-\sum_{j=1}^{r}g(T_{X_j} X,T_{X_j} Y)=\mu g(X,Y),
		\end{align*}
		where $X'$ and $Y'$ are $\pi$-related to $X$ and $Y$ respectively, for any $X,Y \in \Gamma(\mathcal{H})$ and we denote $\mathcal{H}\nabla_X Y=\nabla^{\ast}_X Y$.
		Since  fundamental tensors vanish identically on the manifold (i.e. $T\equiv0, A\equiv0$),  the preceding equation yields
		\begin{align}\label{3.9}
			\frac{1}{2} \Big[g(\nabla^{\ast}_\eta \nabla^{\ast}_X \eta,Y)+g(\nabla^{\ast}_\eta \nabla^{\ast}_Y \eta,X)-g(\nabla^{\ast}_{[\eta,X]}\eta,Y)-g(\nabla^{\ast}_{[\eta,Y]}\eta,X)\\+ 2g(\nabla^{\ast}_X \eta,\nabla^{\ast}_Y \eta)\Big]+ \lambda\Big[g(\nabla^{\ast}_X \eta,Y)+g(X,\nabla^{\ast}_Y \eta)\Big]\notag\\+\operatorname{Ric}'(X',Y')\circ\pi=\mu g(X,Y).\notag
		\end{align}
		On the other hand, $\nabla^{\ast}_\eta \nabla^{\ast}_X \eta,\nabla^{\ast}_\eta \nabla^{\ast}_Y \eta,\nabla^{\ast}_{[\eta,X]}\eta,\nabla^{\ast}_{[\eta,Y]}\eta,\nabla^{\ast}_X \eta$ and $\nabla^{\ast}_Y \eta$
		are the basic vector fields $\pi$-related to  $\nabla'_{\eta'} \nabla'_{X'} \eta',\nabla'_{\eta'} \nabla'_{Y'} \eta',\nabla'_{[\eta',X']} \eta',\nabla'_{[\eta',Y']} \eta',\nabla'_{X'} \eta'$ and $\nabla'_{Y'} \eta'$ respectively. Thus from (\ref{3.9}) we get
		\begin{align*}
			\frac{1}{2}(L_{\eta'} L_{\eta'} g')(X',Y')\circ\pi+\lambda (L_{\eta'} g')(X',Y')\circ\pi+ \operatorname{Ric}'(X',Y')\circ\pi=\mu g'(X',Y')\circ\pi,
		\end{align*}
		which is equivalent to
		\begin{equation}\label{brb}
			\frac{1}{2}(L_{\eta'} L_{\eta'} g')(X',Y')+\lambda (L_{\eta'} g')(X',Y')+\operatorname{Ric}'(X',Y')=\mu g'(X',Y').
		\end{equation}
		Thus the proof is completed.  Since $\eta$ is a unit vector field on the base manifold, taking the trace of \eqref{brb} yields
		${\tau'}=n\mu.$
		Consequently, the rate of change of soliton on the base manifold is classified as {shrinking}, {steady}, or {expanding} according as
		${\tau'}<0,\,
	{\tau'}=0,\,
		\text{or}\,
		{\tau'}>0$, respectively.
	\end{proof}
	\begin{remark}
		Let $\pi:(M,g,\eta,\lambda,\mu)\rightarrow (B,g')$ be a Riemannian submersion between Riemannian manifolds. The total manifold admits a hyperbolic Ricci soliton with the horizontal potential field $\eta$ and vanishing fundamental tensors, then the vector field $N$ is killing on the horizontal distribution $\mathcal{H}$.
	\end{remark}
	\begin{corollary}
		Let $\pi:(M,g,\eta,\lambda,\mu)\rightarrow (B,g')$ be a Riemannian submersion from Riemannian manifold to an Einstein manifold. The total manifold admits a hyperbolic Ricci soliton with the horizontal potential field $\eta$ and vanishing fundamental tensors. If the  vector field $\eta$ is $2$-killing then it is conformal on the horizontal distribution $\mathcal{H}$. 
	\end{corollary}
	\begin{proof}
		Since the total space $(M,g)$ of Riemannian submersion $\pi$ admits a hyperbolic Ricci soliton and vector field $\eta$ is $2$-killing, from (\ref{3.7}) we have
		\begin{align}
			\lambda(L_{\eta} g)(X,Y)
			+\operatorname{Ric}(X,Y)=\mu g(X,Y).\notag
		\end{align}
		Using (\ref{2.13}) in the above equation we get
		\begin{align}
			\lambda(L_{\eta} g)(X,Y)
			+\operatorname{Ric}'(X',Y')\circ\pi+ \notag\frac{1}{2}\Big[g(\nabla_X N,Y)+g(\nabla_Y N,X)\Big]\\
			-2\sum_{i=1}^{n}g(A_X X_i,A_Y X_i)-\sum_{j=1}^{r}g(T_{X_j} X,T_{X_j} Y)=\mu g(X,Y).\notag
		\end{align}
		Since fundamental tensors vanish identically on the manifold, the above equation yields
		\begin{align}
			\lambda(L_{\eta} g)(X,Y) +\operatorname{Ric}'(X',Y')\circ\pi =\mu g(X,Y).\notag
		\end{align}
		Because the base manifold $(B,g')$ is Einstein, the preceding equation implies that the vector field $\eta$ is conformal.
	\end{proof}
	\begin{corollary}
		Let $\pi:(M,g)\rightarrow (B,g')$ be a Riemannian submersion between Riemannian manifolds. If $(B,g')$ admits a hyperbolic Ricci soliton with potential field $\eta'$ which is conformal on the horizontal distribution $\mathcal{H}$ then $(B,g')$ is an Einstein.
	\end{corollary}
	\begin{proof}
		Since $(B,g')$ admits a hyperbolic Ricci soliton  with the potential field $\eta'$, for any $X',Y'\in \Gamma(TB)$ we have
		\begin{equation*}
			\frac{1}{2}(L_{\eta'} L_{\eta'} g')(X',Y')+\lambda (L_{\eta'} g')(X',Y')+\operatorname{Ric}'(X',Y')=\mu g'(X',Y').
		\end{equation*} 
		Since $\eta$ is conformal, from the above equation we get 
		\begin{equation*}
			2 \mu_1^2 g'(X',Y')+2\lambda \mu_1 g'(X',Y')+\operatorname{Ric}'(X',Y')=\mu g'(X',Y'),
		\end{equation*}
		which implies
		\begin{equation*}
			\operatorname{Ric}'(X',Y')=(\mu-2 \mu_1^2 - 2\lambda\mu_1)g'(X',Y').
		\end{equation*}
		Hence, $(B,g')$ is an Einstein.
	\end{proof}
	\begin{theorem}
		Let $\pi:(M,g,\eta,\lambda,\mu)\rightarrow (B,g')$ be a Riemannian submersion between Riemannian manifolds. The total manifold admits a hyperbolic Ricci soliton with the horizontal potential field $\eta$ and vanishing fundamental tensors. If  $\nabla \eta$ is vertical, then  the scalar curvature of any fiber is given by $\mu r$.
	\end{theorem}
	\begin{proof}
		Since the total space $(M,g)$ admits a hyperbolic Ricci soliton  with the horizontal potential field $\eta$, for any $U,W\in\Gamma(\mathcal{V})$ we have
		\begin{equation*}
			\frac{1}{2}(L_\eta L_\eta g)(U,W)+\lambda (L_\eta g)(U,W)+\operatorname{Ric}(U,W)=\mu g(U,W).
		\end{equation*} 
		Using (\ref{2.17}) and (\ref{2.18}) in the above equation, we get
		\begin{align*}
			\frac{1}{2} \Big[g(\nabla_\eta \nabla_U \eta,W)+g(\nabla_\eta \nabla_W \eta,U)-g(\nabla_{[\eta,U]} \eta,W) 
			-g(\nabla_{[\eta,W]} \eta,U)\\\notag+2g(\nabla_U \eta,\nabla_W \eta)\Big] +\lambda\Big[g(\nabla_U \eta,W)+g(U,\nabla_W \eta)\Big]\\
			+\operatorname{Ric}(U,W)=\mu g(U,W).
		\end{align*}
		Again using (\ref{2.12}) in the preceding equation, we obtain
		\begin{align}\label{3.10}
			\frac{1}{2} \Big[g(\nabla_\eta \nabla_U \eta,W)+g(\nabla_\eta \nabla_W \eta,U)-g(\nabla_{[\eta,U]} \eta,W) 
			-g(\nabla_{[\eta,W]} \eta,U)\\\notag+2g(\nabla_U \eta,\nabla_W \eta)\Big] +\lambda\Big[g(\nabla_U \eta,W)+g(U,\nabla_W \eta)\Big]
			+\widehat{\operatorname{Ric}}(U,W)\\-\sum_{j=1}^{r}g(T_U W,T_{U_j}{U_j})+\notag
			\sum_{i=1}^{n}\Big[g(\nabla_{X_i} H,X_i)g(U,W)\\
			+g(A_{X_i}U,A_{X_i}W)\Big]=\mu g(U,W).\notag
		\end{align}
		By tracing (\ref{3.10}), we obtain
		\begin{align*}
			\sum_{j=1}^{r}\Big[ g(\nabla_\eta \nabla_{U_j} \eta,{U_j})-g(\nabla_{[\eta,{U_j}]} \eta,{U_j}) 
			+g(\nabla_{U_j} \eta,\nabla_{U_j} \eta)+2\lambda g(\nabla_{U_j} \eta,{U_j})\\
			+\widehat{\operatorname{Ric}}({U_j},{U_j})-g(T_{{U_j}} {U_j},N) +\sum_{i=1}^{n}\Big[g((\nabla_{X_i}T)_{{U_j}} {U_j},X_i)\\+g(A_{X_i} {U_j},A_{X_i} {U_j})\Big]\Big]=\mu\sum_{j=1}^{r} g({U_j},{U_j}), 
		\end{align*}
		where $\{U_j\}$ is an orthonormal frame on vertical distribution. Using (2.2), (2.3) and (2.4) in the above equation, we get 
		\begin{align*}
			\sum_{j=1}^{r} \Big[ g(A_\eta (\mathcal{H}\nabla_{U_j} \eta),U_j) +g(\mathcal{V}\nabla_\eta (T_{U_j} \eta),U_j) 
			-g( T_{[\eta,{U_j}]}\eta,{U_j})-2\lambda g(N,\eta)\\
			+g(T_{U_j} \eta,T_{U_j} \eta)+g(\mathcal{H}\nabla_{U_j} \eta,\mathcal{H}\nabla_{U_j} \eta)
			+\widehat{\operatorname{Ric}}({U_j},{U_j})-g(T_{{U_j}} {U_j},N)\\ +\sum_{i=1}^{n}\Big[g((\nabla_{X_i}T)_{{U_j}} {U_j},X_i)+g(A_{X_i} {U_j},A_{X_i} {U_j})\Big]\Big]=\mu\sum_{j=1}^{r} g({U_j},{U_j}). 
		\end{align*}
		Since fundamental tensors vanish identically on the manifold and $\nabla \eta$ is vertical, the preceding equation yields 	$\widehat{\tau}=\mu r$, which completes the proof.
	\end{proof}
\begin{theorem}
	Let $\pi:(M,g,\xi,\lambda,\mu)\rightarrow (B,g')$ be a Riemannian submersion between Riemannian manifolds. The total manifold admits a hyperbolic Ricci soliton with the vertical potential field $\xi$ and vanishing fundamental tensors. If  $\nabla \xi$ is horizontal, then  the scalar curvature of base manifold is given by $\mu n$.
		\end{theorem}
	\begin{proof}
			The total space $M$ of Riemannian submersion $\pi$ admits a hyperbolic Ricci soliton  with the vertical potential field $\xi$, therefore by tracing  \eqref{3.6} for any $X,Y\in\Gamma(\mathcal{H})$ we have 
	
		\begin{align}\label{314}
			\frac{1}{2}\sum_{k=1}^{n}\Big[ \Big[g(\nabla_\xi \nabla_{X_k} \xi,X_k)+g(\nabla_\xi \nabla_{X_k} \xi,X_k)-g(\nabla_{[\xi,X_k]} \xi,X_k) 
			-g(\nabla_{[\xi,X_k]} \xi,X_k)\\+2g(\nabla_{X_k} \xi,\nabla_{X_k} \xi)\Big]\notag+\lambda\Big[g(\nabla_{X_k} \xi,X_k)+g(X_k,\nabla_{X_k} \xi)\Big]
			+\operatorname{Ric}'(X'_k,X'_k)\circ\pi\\+\notag\frac{1}{2}\Big[g(\nabla_{X_k} N,X_k)+g(\nabla_{X_k} N,X_k)\Big]
			-2\sum_{i=1}^{n}g(A_{X_k} X_i,A_{X_k} X_i)\\-\sum_{j=1}^{r}g(T_{X_j} X_k,T_{X_j} X_k)\Big]=\mu\sum_{k=1}^{n} g(X_k,X_k),\notag
		\end{align}
		where $\{X_k\}$ is an orthonormal frame on horizontal distribution and $X'_k$  are $\pi$-related to $X_k$ , respectively. Now, using (\ref{2.1}), (\ref{2.2}) and (\ref{2.3}) in (\ref{314}) and get
		\begin{align*}
			\frac{1}{2}\sum_{k=1}^{n}\Big[\Big[ g(\mathcal{H}\nabla_\xi(A_{X_k} \xi),X_k)+g(T_\xi(\mathcal{V}\nabla_{X_k} \xi),X_k)+g(\mathcal{H}\nabla_\xi(A_{X_k} \xi),X_k)+g(T_\xi(\mathcal{V}\nabla_{X_k} \xi),X_k) \\ -g(T_{[\xi,X_k]}\xi,X_k) - g(T_{[\xi,X_k]}\xi,X_k)+2g(A_{X_k} \xi,A_{X_k} \xi)+2g(\mathcal{V}\nabla_{X_k} \xi,\mathcal{V}\nabla_{X_k} \xi)\Big] \\+\lambda\Big[g(A_{X_k} \xi,X_k)+g(X_k,A_{X_k} \xi)\Big]
			+\operatorname{Ric}'(X_k',X_k')\circ\pi+\frac{1}{2}\Big[g(\nabla_{X_k} N,X_k)\\+g(\nabla_{X_k} N,X_k)\Big] 
			-2\sum_{i=1}^{n}g(A_{X_k} X_i,A_{X_k} X_i)-\sum_{j=1}^{r}g(T_{X_j} X_k,T_{X_j} X_k)\Big]=\mu n.
		\end{align*}
		Since fundamental tensors vanish identically on the manifold and $\nabla \xi$ is horizontal, the preceding equation yields 	${\tau'}=\mu n$, which completes the proof.
			\end{proof}
			
	\section{Riemannian submersion with an Almost Hyperbolic Ricci Soliton}
	In this section, we study a Riemannian submersion whose toal manifold admits an almost hyperbolic Ricci soliton with torse-forming potential vector field, and explore the geometric features.
	A vector field ${\zeta}$ on a Riemannian manifold $M$ is said to be torse-forming \cite{ky}, if it satisfies
	\begin{equation}\label{4.1}
		\nabla_E {\zeta}= fE+\phi(E){\zeta}
	\end{equation}
	where $E\in\Gamma(TM)$, $f$ is a function, $\phi$ is a $1$-form and $\nabla$ is the Levi-Civita connection on $M$.
	The vector field is called
	\begin{enumerate}
		\item [(i)] recurrent if $f$ is zero,
		\item [(ii)]concircular if $\phi$ is zero,
		\item[(iii)] concurrent if $\phi$ is zero and $f=1$ is satisfied.
	\end{enumerate}
	A non-flat Riemannian manifold $(M,g)$ is called an mixed generalized quasi-Einstein manifold \cite{am} if its Ricci tensor satisfies
	\begin{equation}\label{4.2}
		\operatorname{Ric}=ag+b\phi\otimes\phi+c\psi\otimes\psi+d(\phi\otimes\psi+\psi\otimes\phi)
	\end{equation}
	for some smooth functions $a,b,c,d$ and $1$-forms $\phi$ and $\psi$. In particular, the manifold $M$ is called
	\begin{enumerate}
		\item [(i)] \cite{mc} generalized quasi-Einstein if $b=0$ or $c=0$,
		\item [(ii)] \cite{mc1} quasi-Einstein if $b=d=0$ or $c=d=0$,
		\item [(iii)] \cite{al} Einstein if $b=c=d=0$.
	\end{enumerate}
	\begin{proposition}\cite{em1}
		Let ${\zeta}$ be a torse-forming  vector field having tangential part $\xi$ and horizontal part $\eta$ on the total manifold $(M,g)$. Then for any vectors $U\in\Gamma(\mathcal{V})$ and $X\in\Gamma(\mathcal{H})$
		\begin{align}\label{4.3}
			\widehat{\nabla}_U \xi &= fU + \phi(U)\xi- T_U \eta,
		\end{align}
		\begin{align} \label{4.4}
			\nabla_X^{\ast} \eta &= fX + \phi(X)\eta- A_X \xi.
		\end{align}
	\end{proposition}	
	\begin{theorem}
		Let $\pi:(M,g,\xi,\lambda,\mu)\rightarrow (B,g')$ be a Riemannian submersion between Riemannian manifolds. The total manifold admits an almost hyperbolic Ricci soliton with the concircular vertical potential field $\xi$ and vanishing fundamental tensors, then any fiber of $\pi$ is an Einstein manifold.
	\end{theorem}
	\begin{proof}
		Since $M$ admits an almost hyperbolic Ricci soliton with vertical potential field $\xi$, from (\ref{3.4}) we have 
		\begin{align*}
			\frac{1}{2} \Big[g(\widehat{\nabla}_\xi \widehat{\nabla}_U \xi,W)+g(\widehat{\nabla}_\xi \widehat{\nabla}_W \xi,U)-g(\widehat{\nabla}_{[\xi,U]} \xi,W)
			-g(\widehat{\nabla}_{[\xi,W]} \xi,U)\\+2g(\widehat{\nabla}_U \xi,\widehat{\nabla}_W \xi)\Big]+ \lambda\Big[g(\widehat{\nabla}_U \xi,W)+
			g(U,\widehat{\nabla}_W \xi)\Big]\\+\operatorname{Ric}(U,W)=\mu g(U,W),
		\end{align*} 
		for any $U,W\in\Gamma(\mathcal{V})$. Now using (\ref{2.12}) in the above equation we get
		\begin{align*}
			\frac{1}{2} \Big[g(\widehat{\nabla}_\xi \widehat{\nabla}_U \xi,W)+g(\widehat{\nabla}_\xi \widehat{\nabla}_W \xi,U)-g(\widehat{\nabla}_{[\xi,U]} \xi,W)
			-g(\widehat{\nabla}_{[\xi,W]} \xi,U)\\+2g(\widehat{\nabla}_U \xi,\widehat{\nabla}_W \xi)\Big]+ \lambda\Big[g(\widehat{\nabla}_U \xi,W)+
			g(U,\widehat{\nabla}_W \xi)\Big]+\widehat{\operatorname{Ric}}(U,W)\\ -g(T_U W,N)+ \sum_{i=1}^{n}\Big[g((\nabla_{X_i}T)_U W,X_i) +g(A_{X_i}U,A_{X_i}W)\Big]=\mu g(U,W).
		\end{align*}
		Since  fundamental tensors vanish identically on the manifold and $\xi$ is concircular vector field, from $(\ref{4.3})$ we have 	$\widehat{\nabla}_U \xi= fU$. Thus from the preceding equation we obtain
		\begin{align*}
			\frac{1}{2} \Big[2\xi(f)g(U,W)+fg(\widehat{\nabla}_\xi U,W)+fg(\widehat{\nabla}_\xi W,U)-fg([\xi,U],W) 
			-fg([\xi,W],U)\\+2f^2g(U,W)\Big]+2\lambda fg(U,W)+\widehat{\operatorname{Ric}}(U,W)=\mu g(U,W). 
		\end{align*}
		Further, which can be written as
		\begin{align*}
			\widehat{\operatorname{Ric}}(U,W)=(\mu-2\lambda f-\xi(f)-2f^2) g(U,W), 
		\end{align*}
		Hence any fiber of $\pi$ is an Einstein.
	\end{proof}
	\begin{theorem}
		Let $\pi:(M,g,\eta,\lambda,\mu)\rightarrow (B,g')$ be a Riemannian submersion between Riemannian manifolds. The total manifold admits an almost hyperbolic Ricci soliton with the concircular horizontal potential field $\eta$ and vanishing fundamental tensors, then the base manifold $(B,g')$ is an Einstein manifold.
	\end{theorem}
	\begin{proof}
		Since $M$ admits an almost hyperbolic Ricci soliton with horizontal potential field $\eta$ and fundamental tensors vanish identically on the manifold, from (\ref{3.9}) we have 
		\begin{align*}
			\frac{1}{2} \Big[g(\nabla^{\ast}_\eta \nabla^{\ast}_X \eta,Y)+g(\nabla^{\ast}_\eta \nabla^{\ast}_Y \eta,X)-g(\nabla^{\ast}_{[\eta,X]} \eta,Y)
			-g(\nabla^{\ast}_{[\eta,Y]} \eta,X)\\+2g(\nabla^{\ast}_X \eta,\nabla^{\ast}_Y \eta)\Big]+ \lambda\Big[g(\nabla^{\ast}_X \eta,Y)+
			g(X,\nabla^{\ast}_Y \eta)\Big]\\+\operatorname{Ric}'(X',Y')\circ\pi=\mu g(X,Y).
		\end{align*} 
		Also $\eta$ is concircular vector field, from $(\ref{4.4})$ we have 	$\nabla^{\ast}_X \eta= fX$. Thus from the preceding equation we obtain
		\begin{align*}
			\frac{1}{2} \Big[2\eta(f)g(X,Y)+fg(\nabla^{\ast}_\eta X,Y)+fg(\nabla^{\ast}_\eta Y,X)-fg([\eta,X],Y) 
			-fg([\eta,Y],X)\\ +2f^2g(X,Y)\Big]+2\lambda fg(X,Y)+\operatorname{Ric}'(X',Y')\circ\pi=\mu g(X,Y). 
		\end{align*}
		Further, the above equation can be written as
		\begin{align*}
			\operatorname{Ric}'(X',Y')\circ\pi=(\mu-2\lambda f-\eta(f)-2f^2) g(X,Y), 
		\end{align*}
		which is equivalent to
		\begin{align*}
			\operatorname{Ric}'(X',Y')=(\mu-2\lambda f-\eta(f)-2f^2) g'(X',Y'). 
		\end{align*}
		Hence, $(B,g')$ is an Einstein.
	\end{proof}
	\begin{theorem}
		Let $\pi:(M,g,\xi,\lambda,\mu,\phi)\rightarrow (B,g')$ be a Riemannian submersion between Riemannian manifolds. The total manifold admits an almost hyperbolic Ricci soliton with the torse-forming vertical potential field $\xi$ having parallel $1$-form $\phi$ and totally geodesic fibers, then the extrinsic vertical distribution  is an almost generalized quasi-Einstein (i.e. extrinsic vertical Ricci  tensor $\operatorname{Ric}|_{\mathcal{V}}$ satisties (\ref{4.2})).
	\end{theorem}
	\begin{proof}
		Since	$M$ admits an almost hyperbolic Ricci soliton with the vertical potential field $\xi$, for any  $U,W\in\Gamma(\mathcal{V})$ from (\ref{3.4}) we have 
		\begin{align*}
			\frac{1}{2} \Big[g(\widehat{\nabla}_\xi \widehat{\nabla}_U \xi,W)+g(\widehat{\nabla}_\xi \widehat{\nabla}_W \xi,U)-g(\widehat{\nabla}_{[\xi,U]} \xi,W)
			-g(\widehat{\nabla}_{[\xi,W]} \xi,U)\\+2g(\widehat{\nabla}_U \xi,\widehat{\nabla}_W \xi)\Big]+ \lambda\Big[g(\widehat{\nabla}_U \xi,W)+ g(U,\widehat{\nabla}_W \xi)\Big]\\+\operatorname{Ric}(U,W)=\mu g(U,W).
		\end{align*}
		Now $\xi$ is torse-forming vector field, using $(\ref{4.3})$ in the above equation we obtain
		\begin{align*}
			\frac{1}{2}\Big[2\xi(f)g(U,W)+2\phi (U)\phi(\xi)g(\xi,W) +\phi (U)g(\eta,T_\xi W)
			-g((\widehat{\nabla}_\xi T)_U \eta,W)\\ +2\phi (W)\phi(\xi)g(\xi,U)+g((\widehat{\nabla}_\xi \phi)(U) \xi,W)
			+ \phi (W)g(\eta,T_\xi U)-g((\widehat{\nabla}_\xi T)_W \eta,U)\\
			-\phi(T_U \eta)g(\xi,W)-g(T_{W}\eta,\widehat{\nabla}_U \xi) -\phi(T_W \eta)g(\xi,U)-g(T_{U}\eta,\widehat{\nabla}_W \xi)
			\\	+4f^2g(U,W)+5f\phi(W)g(U,\xi)+g((\widehat{\nabla}_\xi \phi)(W) \xi,U)
			+4fg(T_U W, \eta) \\+5f\phi(U)g(\xi,W)
			+2\phi(U)\phi(W)||\xi||^2+2\phi(U)g(T_W \xi,\eta)+2fg(\eta,T_U W)\\
			+2\phi(W)g(\eta,T_U \xi) +2g(T_U \eta,T_W \eta)\Big]+2\lambda fg(U,W)+\lambda\phi(U)g(\xi,W)\\
			+2\lambda g(\eta,T_U W) +\lambda\phi(W)g(\xi,U)+\operatorname{Ric}(U,W)=\mu g(U,W). 
		\end{align*}
		Also $\phi$ is parallel along vertical potential field (i.e. $\nabla\phi=0$), the above equation becomes
		\begin{align}\label{4.5}
			\Big[\frac{5}{2}f+\phi(\xi)+\lambda\Big]\Big[\phi (U)g(\xi,W)+\phi(W)g(\xi,U)\Big] -\frac{3}{2}\Big[\phi (U)g(T_\xi \eta,W)\\\notag + \phi (W)g(T_\xi \eta,U)\Big]
			+g((\widehat{\nabla}_\xi T)_U W,\eta) -\frac{1}{2}\Big[\phi(T_U \eta)g(\xi,W)+\phi(T_W \eta)g(\xi,U)\\\notag +g(T_{W}\eta,\widehat{\nabla}_U \xi)+g(T_{U}\eta,\widehat{\nabla}_W \xi)\Big]
			+\phi(U)\phi(W)||\xi||^2 +g(T_U \eta,T_W \eta)+\\(3f+2\lambda )g(\eta,T_U W)+\operatorname{Ric}(U,W)
			=(\mu-\xi(f)-2f^2-2\lambda f)g(U,W).\notag 
		\end{align}
		The total manifold is totally geodesic, therfore from the preceding equation we get
		\begin{align*}
			\Big[\frac{5}{2}f+\phi(\xi)+\lambda\Big]\Big[\phi (U)g(\xi,W)+\phi(W)g(\xi,U)\Big]
			+||\xi||^2\phi(U)\phi(W)\\ +\operatorname{Ric}(U,W)=(\mu-\xi(f)-2f^2-2\lambda f)g(U,W).\notag 
		\end{align*}
		Let the dual $1$-form of $\xi$ be $\alpha$, then $\alpha(U)=g(\xi,U)$. Thus we obtain
		\begin{align*}
			\operatorname{Ric}(U,W)|_{\mathcal{V}}=(\mu-\xi(f)-2f^2-2\lambda f)g(U,W)-||\xi||^2\phi(U)\phi(W)\\ -\Big[\frac{5}{2}f+\phi(\xi)+\lambda\Big]\Big[\phi (U)\alpha(W)+\phi(W)\alpha(U)\Big] , 
		\end{align*}
		which is equivalent to
		\begin{align*}
			\operatorname{Ric}|_{\mathcal{V}}=(\mu-\xi(f)-2f^2-2\lambda f)g -||\xi||^2\phi\otimes\phi \\-\Big[\frac{5}{2}f+\phi(\xi)+\lambda\Big]\Big[\phi\otimes\alpha+\alpha\otimes\phi\Big].
		\end{align*}
		Thus  the extrinsic vertical distribution is an almost generalized quasi-Einstein.
	\end{proof}
	\begin{theorem}
		Let $\pi:(M,g,\xi,\lambda,\mu,\phi)\rightarrow (B,g')$ be a Riemannian submersion between Riemannian manifolds. The total manifold admits an almost hyperbolic Ricci soliton with the torse-forming vertical potential field $\xi$ having parallel $1$-form $\phi$ and vanishing fundamental tensors, then the any fiber is an almost generalized quasi-Einstein.
	\end{theorem}
	\begin{proof}
		Since  fundamental tensors vanish identically on the manifold, from equation (\ref{4.5}) we get 
		\begin{align*}
			\widehat{\operatorname{Ric}}(U,W)=(\mu-\xi(f)-2f^2-2\lambda f)g(U,W)-||\xi||^2\phi(U)\phi(W)\\ -\Big[\frac{5}{2}f+\phi(\xi)+\lambda\Big]\Big[\phi (U)\alpha(W)+\phi(W)\alpha(U)\Big] , 
		\end{align*}
		for any $U,W\in\Gamma(\mathcal{V})$,	which implies
		\begin{align*}
			\widehat{\operatorname{Ric}}=(\mu-\xi(f)-2f^2-2\lambda f)g -||\xi||^2\phi\otimes\phi \\-\Big[\frac{5}{2}f+\phi(\xi)+\lambda\Big]\Big[\phi\otimes\alpha+\alpha\otimes\phi\Big].
		\end{align*}
		Hence any fiber is an almost generalized quasi-Einstein.
	\end{proof}
	\begin{theorem}
		Let $\pi:(M,g,\eta,\lambda,\mu,\phi)\rightarrow (B,g')$ be a Riemannian submersion between Riemannian manifolds. The total manifold admits an almost hyperbolic Ricci soliton with the torse-forming horizontal potential field $\eta$ having parallel $1$-form $\phi$ and integrable horizontal distribution, then the  extrinsic horizontal distribution is an almost generalized quasi-Einstein (i.e. extrinsic horizontal Ricci tensor $\operatorname{Ric}|_{\mathcal{H}}$ satisties (\ref{4.2})).
	\end{theorem}
	\begin{proof}
		Since $M$ admits an almost hyperbolic Ricci soliton with the horizontal potential field $\eta$,  for any $X,Y\in\Gamma(\mathcal{H})$ from (\ref{3.8}) we have  
		\begin{align*}
			\frac{1}{2} \Big[g(\nabla^{\ast}_\eta \nabla^{\ast}_X \eta,Y)+g(\nabla^{\ast}_\eta \nabla^{\ast}_Y \eta,X)-g(\mathcal{H}\nabla_{[\eta,X]} \eta,Y)
			-g(\mathcal{H}\nabla_{[\eta,Y]} \eta,X)\\ \notag+2g(\nabla^{\ast}_X \eta,\nabla^{\ast}_Y \eta)
			+4g(A_X \eta,A_Y \eta)\Big]+ \lambda\Big[g(\nabla^{\ast}_X \eta,Y)+ g(X,\nabla^{\ast}_Y \eta)\Big]\\+\operatorname{Ric}(X,Y)=\mu g(X,Y)\notag. 
		\end{align*}
		Now $\eta$ is torse-forming vector field, using (\ref{4.4}) in the above equation we obtain
		\begin{align*}
			\frac{1}{2}\Big[2\eta(f)g(X,Y)+2\phi (X)\phi(\eta)g(\eta,Y) +\phi (X)g(\xi,A_\eta Y)
			-g((\nabla^{\ast}_\eta A)_X \xi,Y)\\+2\phi (Y)\phi(\eta)g(\eta,X)+g((\nabla^{\ast}_\eta \phi)(X) \eta,Y)+ \phi (Y)g(\xi,A_\eta X)-g((\nabla^{\ast}_\eta A)_Y \xi,X)\\-\phi(A_X \xi)g(\eta,Y) +g(A_{Y}\xi,\nabla^{\ast}_X \eta) -2\phi(A_\eta X)g(\eta,Y)-2\phi(A_\eta Y)g(\eta,X)-\\\phi(A_Y \xi)g(\eta,X)+g(A_{X}\xi,\nabla^{\ast}_Y \eta) +4f^2g(X,Y)+5f\phi(Y)g(X,\eta)\\+g((\nabla^{\ast}_\eta \phi)(Y) \eta,X)+4fg(A_X Y, \xi) +5f\phi(X)g(\eta,Y)+2\phi(X)\phi(Y)||\eta||^2\\
			+2\phi(X)g(A_Y \eta,\xi)+2fg(\xi,A_X Y)
			+2\phi(Y)g(\xi,A_X \eta) +2g(A_X \xi,A_Y \xi)\\+4g(A_X \eta,A_Y \eta)\Big]
			+2\lambda fg(X,Y)+\lambda\phi(X)g(\eta,Y)
			+2\lambda g(\xi,A_X Y)\\ +\lambda\phi(Y)g(\eta,X)+\operatorname{Ric}(X,Y)=\mu g(X,Y). 
		\end{align*}
		Also $\phi$ is parallel along horizontal potential field (i.e. $\nabla\phi=0$), from the above equation we get
		\begin{align}\label{4.6}
			\Big[\frac{5}{2}f+\phi(\eta)+\lambda\Big]\Big[\phi (X)g(\eta,Y)+\phi(Y)g(\eta,X)\Big]-\frac{3}{2}\Big[\phi (X)g(A_\eta \xi,Y)\\ \notag+\phi (Y)g(A_\eta \xi,X)\Big] -\frac{1}{2} \Big[\phi(A_X \xi)g(\eta,Y)+\phi(A_Y \xi)g(\eta,X) -g(A_{Y}\xi,\nabla^{\ast}_X \eta)\\\notag-g(A_{X}\xi,\nabla^{\ast}_Y \eta)\Big]
			+\phi(X)\phi(Y)||\eta||^2+2g(A_X \eta,A_Y \eta) +g(A_X \xi,A_Y \xi)\\+(3f+2\lambda )g(\xi,A_X Y)+\operatorname{Ric}(X,Y)
			=(\mu-\eta(f)-2f^2-2\lambda f)g(X,Y). \notag
		\end{align}
		The horizontal distribution $\mathcal{H}$ is integrable, therefore from preceding equation we deduce
		\begin{align*}
			\Big[\frac{5}{2}f+\phi(\eta)+\lambda\Big]\Big[\phi (X)g(\eta,Y)+\phi(Y)g(\eta,X)\Big]
			+||\eta||^2\phi(X)\phi(Y) \\+\operatorname{Ric}(X,Y)=(\mu-\eta(f)-2f^2-2\lambda f)g(X,Y).\notag 
		\end{align*}
		Let the dual $1$-form of $\eta$ be $\beta$, i.e.  $\beta(X)=g(\eta,X)$. Thus
		\begin{align*}
			\operatorname{Ric}(X,Y)|_{\mathcal{H}}=(\mu-\eta(f)-2f^2-2\lambda f)g(X,Y)-||\eta||^2\phi(X)\phi(Y)\\ -\Big[\frac{5}{2}f+\phi(\eta)+\lambda\Big]\Big[\phi (X)\beta(Y)+\phi(Y)\beta(X)\Big], 
		\end{align*}
		which is equivalent to
		\begin{align*}
			\operatorname{Ric}|_{\mathcal{H}}=(\mu-\eta(f)-2f^2-2\lambda f)g -||\eta||^2\phi\otimes\phi \\-\Big[\frac{5}{2}f+\phi(\eta)+\lambda\Big]\Big[\phi\otimes\beta+\beta\otimes\phi\Big].
		\end{align*}
		Thus the  extrinsic horizontal distribution
		is an almost generalized quasi-Einstein.
	\end{proof}
	
	\begin{theorem}
		Let $\pi:(M,g,\eta,\lambda,\mu,\phi)\rightarrow (B,g')$ be a Riemannian submersion between Riemannian manifolds. The total manifold admits an almost hyperbolic Ricci soliton with the torse-forming horizontal potential field $\eta$ having parallel $1$-form $\phi$ and vanishing fundamental tensors, then the base manifold is an almost generalized quasi-Einstein.  
	\end{theorem}
	\begin{proof}
		Since  fundamental tensors vanish identically on the manifold, from (\ref{4.6}) we get 
		\begin{align*}
			\Big[\frac{5}{2}f+\phi(\eta)+\lambda\Big]\Big[\phi (X)g(\eta,Y)+\phi(Y)g(\eta,X)\Big]
			+||\eta||^2\phi(X)\phi(Y) \\+\operatorname{Ric}'(X',Y')\circ\pi=(\mu-\eta(f)-2f^2-2\lambda f)g(X,Y). 
		\end{align*}
		Let the dual $1$-form of $\eta$ be $\beta$, i.e. $\beta(X)=g(\eta,X)$. Thus
		\begin{align*}
			\operatorname{Ric}'(X',Y')\circ\pi=(\mu-\eta(f)-2f^2-2\lambda f)g(X,Y)-||\eta||^2\phi(X)\phi(Y)\\ -\Big[\frac{5}{2}f+\phi(\eta)+\lambda\Big]\Big[\phi (X)\beta(Y)+\phi(Y)\beta(X)\Big]. 
		\end{align*}
		Since the horizontal vector fields $X,Y$ are $\pi$ related to $X',Y'$ on $M'$ respectively, the preceding equation yields
		\begin{align*}
			\operatorname{Ric}'=(\mu-\eta(f)-2f^2-2\lambda f)g' -||\eta||^2\phi'\otimes\phi' \\-\Big[\frac{5}{2}f+\phi(\eta)+\lambda\Big]\Big[\phi'\otimes\beta'+\beta'\otimes\phi'\Big],
		\end{align*}
		where $\phi'$ and $\beta'$ are $1$-forms on $M'$ such that $\phi(X)=\phi'(X')\circ\pi$ and 
		$\beta(X)=\beta'(X')\circ\pi$.
		Hence the base manifold is an almost generalized quasi-Einstein.  
	\end{proof}
	\section{Examples}
	In this section we prove the existence of a hyperbolic Ricci soliton and an almost hyperbolic Ricci soliton on Riemannian submersion by constructing some examples.
	\begin{example}
		We consider the manifold $M=\{(x_1,x_2,x_3)\in R^3, x_3 \neq0\}$ equipped with Riemannian metric $g=\frac{1}{x_3^2}\Big\{(dx_1)^2+(dx_2)^2+(dx_3)^2\Big\}$ and the manifold $B=\{(y_1,y_2)\in R^2, y_2 \neq0\}$ endowed with the Riemannian metric $g'=\frac{1}{y_2^2}\Big\{(dy_1)^2+(dy_2)^2\Big\}$.
		
		Define a submersion $\pi:(M,g)\rightarrow(B,g')$ by
		\begin{align*}
			\pi(x_1,x_2,x_3)= \left(\frac{x_1-x_2}{\sqrt{2}},x_3\right).
		\end{align*}
		Then we obtain 
		\begin{align*}
			Ker\pi_{\ast}&= span\left\{ v_1=\frac{\partial}{\partial x_1}+\frac{\partial}{\partial x_2}\right\},\\
			(Ker\pi_{\ast})^{\bot}&= span\left\{ u_1=\frac{\partial}{\partial x_1}-\frac{\partial}{\partial x_2},u_2=\frac{\partial}{\partial x_3}\right\},
		\end{align*}
		Since $\pi_{\ast}$ preserves the metric of horizontal space
		(i.e.	$g(u_i,u_i)=g'(\pi_{\ast}u_i,\pi_{\ast}u_i)$; $i=1,2$), thus $\pi$ is a Riemannian submersion. Now we show that the total space $M$ admits a hyperbolic Ricci soliton. The vector fields are
		\begin{align*}
			e_1=x_3\frac{\partial}{\partial_{x_1}},\ \ e_2=x_3\frac{\partial}{\partial_{x_2}},\ \ e_3=x_3\frac{\partial}{\partial_{x_3}},
		\end{align*}
		and the Riemannian metric $g$ as
		\begin{align*}
			g(e_i,e_j) =
			\begin{cases}
				1, & i=j,\\
				0, & otherwise.
			\end{cases}
		\end{align*}
		
		Now, using the Koszul's formula the Levi-civita connection $\nabla$ is obtained and 
		by using the definition of Riemannian curvature, we obtain the Ricci tensor 
		\[\operatorname{Ric}(e_i,e_j)=\left[
		\begin{matrix}
			-2 & 0 & 0\\
			0 & -2 & 0\\
			0 & 0 & -2
		\end{matrix}
		\right]\]
		Let the potential vector field ${\zeta}=e_3$, then the Lie-derivative of vector fields as
		\[L_{\zeta} g(e_i,e_j)=\left[
		\begin{matrix}
			-2 & 0 & 0\\
			0 & -2 & 0\\
			0 & 0 & 0
		\end{matrix}
		\right]\]
		and the second Lie-derivative is calculated as
		\[L_{\zeta} L_{\zeta} g(e_i,e_j)=\left[
		\begin{matrix}
			4 & 0 & 0\\
			0 & 4 & 0\\
			0 & 0 & 0
		\end{matrix}
		\right]\]
		Thus $(M,g)$ admits a hyperbolic Ricci soliton with potential vector field $\zeta=e_3$ and the Riemannian metric $g$ is an expanding hyperbolic Ricci soliton with soliton constant $\lambda = 1$ and shrinking rate of change of soliton $\mu =-2$.
	\end{example}
	
	\begin{example}
		We consider the manifold $M=\{(x_1,x_2,x_3,x_4,x_5)\in R^5, x_5 \neq0\}$ equipped with Riemannian metric $g=\frac{1}{x_5^4}\Big\{(dx_1)^2+(dx_2)^2+(dx_3)^2+(dx_4)^2+(dx_5)^2\Big\}$ and the manifold $B=\{(y_1,y_2)\in R^2, y_2 \neq0\}$ endowed with the Riemannian metric $g'=\frac{1}{y_2^4}\Big\{(dy_1)^2+(dy_2)^2\Big\}$.
		
		Define a submersion $\pi:(M,g)\rightarrow(B,g')$  by
		\begin{align*}
			\pi(x_1,x_2,x_3,x_4,x_5)=\left(\frac{x_1+x_2}{\sqrt{2}},x_5\right)
		\end{align*}
		Then we obtain
		\begin{align*}
			Ker\pi_{\ast}&= span\left\{ v_1=\frac{\partial}{\partial x_1}-\frac{\partial}{\partial x_2}, v_2=\frac{\partial}{\partial x_3}, v_3=\frac{\partial}{\partial x_4}\right\},\\
			(Ker\pi_{\ast})^{\bot}&= span\left\{ u_1=\frac{\partial}{\partial x_1}+\frac{\partial}{\partial x_2},u_2=\frac{\partial}{\partial x_5}\right\},
		\end{align*}
		Since  $\pi_{\ast}$ preserves the metric of horizontal space 
		(i.e.	$g(u_i,u_i)=g'(\pi_{\ast}u_i,\pi_{\ast}u_i)$; $i=1,2$), thus $\pi$ is a Riemannian submersion. Now we show that the total space $M$ admits an almost hyperbolic Ricci soliton. The vector fields are
		\begin{align*}
			e_1=x_5^2\frac{\partial}{\partial_{x_1}},\ \ e_2=x_5^2\frac{\partial}{\partial_{x_2}},\ \ 
			e_3=x_5^2\frac{\partial}{\partial_{x_3}},\ \
			e_4=x_5^2\frac{\partial}{\partial_{x_4}},\ \
			e_5=x_5^2\frac{\partial}{\partial_{x_5}},
		\end{align*}
		and the Riemannian metric $g$ as
		\begin{align*}
			g(e_i,e_j) =
			\begin{cases}
				1, & i=j,\\
				0, & otherwise.
			\end{cases}
		\end{align*}
		
		Now, using Koszul's formula the Levi-civita connection $\nabla$ is obtained and by using the definition of Riemannian curvature, we obtain the Ricci tensor 
		\[\operatorname{Ric}(e_i,e_j)=\left[
		\begin{matrix}
			-16x_5^2 & 0 & 0 & 0 & 0\\
			0 & -16x_5^2 & 0 & 0 & 0\\
			0 & 0 & -16x_5^2 & 0 & 0\\
			0 & 0 & 0 & -16x_5^2 & 0\\
			0 & 0 & 0 & 0 &-16x_5^2
		\end{matrix}
		\right]\]
		Let the potential vector field $\zeta=e_5$, then the Lie-derivative of vector fields as
		\[L_{\zeta} g(e_i,e_j)=\left[
		\begin{matrix}
			-4x_5 & 0 & 0 & 0 & 0\\
			0 & -4x_5 & 0 & 0 & 0\\
			0 & 0 & -4x_5 & 0 & 0\\
			0 & 0 & 0 & -4x_5 & 0\\
			0 & 0 & 0 & 0 & 0
		\end{matrix}
		\right]\]
		and the second Lie-derivative is calculated as
		\[L_{\zeta} L_{\zeta} g(e_i,e_j)=\left[
		\begin{matrix}
			12x_5^2 & 0 & 0 & 0 & 0\\
			0 & 12x_5^2 & 0 & 0 & 0\\
			0 & 0 & 12x_5^2 & 0 & 0\\
			0 & 0 & 0 & 12x_5^2 & 0\\
			0 & 0 & 0 & 0 & 0
		\end{matrix}
		\right]\]
		Thus $(M,g)$ admits an almost hyperbolic Ricci soliton with potential vector field $\zeta=e_5$ and the Riemannian metric $g$ is  an expanding almost hyperbolic Ricci soliton with soliton constant $\lambda =\frac{3}{2}x_5$ and shrinking rate of change of soliton $\mu =-16x_5^2$.
	\end{example}
	\section{Conclusion}
	After a comprehensive study, we conclude that if a Riemannian submersion has a total manifold admitting a hyperbolic Ricci soliton, then under suitable conditions the fibers and the base manifold inherit hyperbolic or almost hyperbolic Ricci soliton structures. Moreover, if the total space admits an almost hyperbolic Ricci soliton with a torse-forming vector field, then both the fibers and the base manifold become almost generalized quasi-Einstein manifolds.
	
	Riemannian submersions admitting hyperbolic-type Ricci solitons are important in geometry and mathematical physics, particularly in understanding Type III singularities, long-time behavior of geometric flows, negatively curved manifolds, and stability of hyperbolic metrics. They also arise in self-similar solutions of the Einstein equations, expanding cosmological models, asymptotically hyperbolic metrics, AdS geometry, entropy functionals, nonlinear PDE analysis, and stability theory.
	
	This study can be further extended to other classes of submersions, such as semi-invariant, slant, semi-slant, conformal, and Clairaut submersions, as well as to maps defined on manifolds with different differential structures and solitons.
	
	\textbf {Acknowledgement} : Akhilesh Yadav gratefully acknowledges the financial support provided by BHU, Ref. No. R/Dev/D/IoE/Incentive (Phase IV)/2024-25/82487. Sonali Jaiswal acknowledges the financial assistance provided by the University Grants Commission (U.G.C.), India, in support of her research work.
	
	\textbf {Data Availability Statement} : Data sharing is not applicable to this article as no datasets were generated or analysed during the study.
	
	\textbf {Conflict of Interest} : The authors declare that they have no conflict of interest.

	\noindent\author{Sonali Jaiswal}\\
	\date{Department of Mathematics, Institute of Science, \\Banaras Hindu University, Varanasi-221005, India}\\
	\maketitle {\noindent E-mail: jais.sona94@gmail.com}
	
	\noindent\author{Akhilesh Yadav}\\
	\date{Department of Mathematics, Institute of Science, \\Banaras Hindu University, Varanasi-221005, India}\\
	\maketitle {\noindent E-mail: akhilesha68@gmail.com}
	
	\noindent\author{S.S. Shukla}\\
	\date{Department of Mathematics, Institute of Science, \\University of Allahabad, Allahabad-211002, India}\\
	\maketitle {\noindent E-mail: ssshukla\_au@rediffmail.com}

\end{document}